\documentclass[11pt,reqno,a4paper]{amsart}
\usepackage[margin=1in]{geometry}
\usepackage[notref,notcite]{}
\usepackage[usenames]{color}
\usepackage{amsmath,pdfsync,verbatim,graphicx,epstopdf,enumerate}
\usepackage[colorlinks=true]{hyperref}
\usepackage{cancel}

\usepackage[framemethod=tikz]{mdframed}
\hypersetup{allcolors=blue}
\allowdisplaybreaks
\numberwithin{equation}{section}
\newcommand{\I}{\mathrm{i}}

\newcommand{\wh}{\widehat}

\newcommand{\lb}{\left(}

\newcommand{\rb}{\right)}
\newcommand{\PD}{\partial}

\newcommand{\Beq}{\begin{equation}}
    \newcommand{\Eeq}{\end{equation}}
\newcommand{\beq}{\begin{equation*}}
    \newcommand{\eeq}{\end{equation*}}
\newcommand{\bal}{\begin{align}}
    \newcommand{\eal}{\end{align}}

\renewcommand{\L}{\langle}

\usepackage{mathtools}

\newcommand{\bp}{\begin{prob}}
    \newcommand{\ep}{\end{prob}}
\newcommand{\bpr}{\begin{proof}}
    \newcommand{\epr}{\end{proof}}

\newcommand{\bel}[1]{\begin{equation}\label{#1}}
    \newcommand{\ee}{\end{equation}}

\newtheorem{theorem}{Theorem}[section]

\newtheorem{lemma}[theorem]{Lemma}

\theoremstyle{definition}
\newtheorem{definition}[theorem]{Definition}

\newcommand{\Rn}{\mathbb{R}^n}
\newcommand{\R}{\rangle}
\newcommand{\D}{\mathrm{d}}

\newcommand{\Rb}{\mathbb{R}}

\newcommand{\vp}{\varphi}

\usepackage[usenames]{color}
\usepackage{amsmath,pdfsync,verbatim,graphicx,epstopdf,enumerate}

\newcommand{\wt}{\widetilde}

\newcommand{\Sc}{\mathcal{S}}

\newcommand{\Sb}{\mathbb{S}}

\renewcommand{\L}{\langle}

\title[Range characterization in 2D]{Range characterization of the ray transform on Sobolev spaces of symmetric tensor fields in two dimensions}
\author[Agrawal, Krishnan and Sharafutdinov]{Divyansh Agrawal$^\ast$, Venkateswaran P. Krishnan$^\ast$ and Vladimir A. Sharafutdinov$^\dagger$}

\address {$^\ast$~Centre for Applicable Mathematics, Tata Institute of Fundamental Research, India
\newline
E-mail:{\tt\  agrawald@tifrbng.res.in, vkrishnan@tifrbng.res.in}}
\thanks{The first author was supported by SERB Overseas Visiting Doctoral Fellowship.}
\thanks{The second author acknowledges the support of the Department of Atomic Energy,  Government of India, under
Project No.\ 12-R\&D-TFR-5.01-0520. }
\address{$^\dagger$ Sobolev Institute of Mathematics, 4 Koptyug Av., 630090, Novosibirsk, Russia
\newline
E-mail:{\tt\ sharaf@math.nsc.ru}}
\thanks{The work of the third author was performed according to the Government research assignment
for IM SB RAS, project FWNF-2026-0026.}

\date{\today}

\begin{document}
\begin{abstract}
The ray transform $I_m$ integrates a symmetric $m$ rank tensor field
$f$ on $\Rb^n$ over lines. In the case of $n\ge3$, the range
characterization of the operator  $I_m$ on weighted
Sobolev spaces $H^{s}_t({{\mathbb R}}^n;S^m{{\mathbb R}}^n)$ was
obtained in [V. Krishnan and V. Sharafutdinov. Range
characterization of ray transform on Sobolev spaces of symmetric
tensor fields. Inverse Problems and Imaging, 18(6), 1272--1293,
2024]. Here we obtain a range characterization  result in higher order weighted Sobolev spaces in two dimensions. Range characterization in the case of $n=2$ is very different  from that for
$n\ge3$, and this allows us to obtain such a result in higher order weighted Sobolev spaces $H^{r,s}_t(\Rb^2)$ for any real $r$. Nevertheless, our main tool is again  the Reshetnyak
formula stating that $\lVert
I_mf\rVert_{H^{(r,s+1/2)}_{t+1/2}(T{{\mathbb S}}^{n-1})}=\lVert
f\rVert_{H^{(r,s)}_t({{\mathbb R}}^n;S^m{{\mathbb R}}^n)}$ for a
solenoidal tensor field $f$.

\end{abstract}

\subjclass[2020]{Primary: 44A12, 65R32. Secondary: 46F12.}
\keywords{Ray transform, Reshetnyak formula, range characterization, inverse problems, tensor analysis.}

\maketitle
\section{Introduction}

The ray transform $I_m$ on the Euclidean space integrates rank $m$
symmetric tensor fields (integrates functions in the case of $m=0$)
over lines. This transform arises in several applications such as
computerized tomography ($m=0$), Doppler tomography ($m=1$), travel
time tomography ($m=2$ and $m=4$) and polarization tomography to
name a few. A closely related operator is the Radon transform that
integrates functions over hyperplanes. In two dimensions, the
operator $I_0$ coincides with the Radon transform up to notations.
These transforms are well-studied, see
\cite{Helgason:Book,Natterer:Book,Sharafutdinov:Book}. We are
interested in the range characterization of the ray transform $I_m$
on weighted Sobolev spaces in two dimensions. A special case of the corresponding
result in dimensions $n\geq 3$ was obtained in the recent work
\cite{KS2}, where the 2D case was posed as an open question. The
question is answered in the current paper.

In the range characterization of the Radon transform on the Schwartz
space ${\mathcal S}({\mathbb R}^n)$, the so called Gel'fand --
Helgason -- Ludwig (GHL) integral conditions play the main role.
These conditions disappear while passing from the Schwartz space to
$L^2({\mathbb R}^n)$ \cite{{GGV_Book},Helgason:Book} and to Sobolev
spaces \cite{Sh3}.

In the range characterization of the ray transform $I_m$ on the
Schwartz space ${\mathcal S}({\mathbb R}^n;S^m{\mathbb R}^n)$ in
dimensions $n\ge3$, the John differential equations play the main
role. In the case of $n=3$ and $m=0$ there is one second order John
equation discovered in the pioneering
work \cite{J} by F.~John. In the
case of $m=0$ and arbitrary $n\ge3$, there is a system of second
order John equations \cite{Helgason:Book}. For arbitrary $m\ge0$ and
$n\ge3$, there is a system of $2(m+1)$ order John's differential
equations \cite{Sharafutdinov:Book}.  John's equations survive while
passing from the Schwartz space to Sobolev spaces if the equations
are treated in the distribution sense \cite{KS2}.

The situation is very different for the ray transform $I_m$ in two
dimensions. There are some GHL type integral conditions in the range
characterization of $I_m$ on the Schwartz space in the 2D case
\cite{P}. Nevertheless, as is shown in the current work, neither
integral GHL conditions nor John's differential equations survive
while passing from the Schwartz space to Sobolev spaces in two
dimensions.

\section{Preliminaries and statements of the main results}

Let $S^{m}{\mathbb R}^{n}$ be the  ${n+m-1}\choose m$-dimensional
complex vector space of rank $m$ symmetric tensors on ${{\mathbb
R}}^n$. Let ${\mathcal S}({\mathbb R}^{n}; S^{m}{\mathbb
R}^{n})={\mathcal S}({\mathbb R}^{n})\otimes S^{m}{\mathbb R}^{n}$
denote the Schwartz space of $S^{m}{\mathbb R}^{n}$-valued functions
on ${\mathbb R}^n$ equipped with the standard topology. Elements of
${\mathcal S}({\mathbb R}^{n}; S^{m}{\mathbb R}^{n})$ are smooth
fast decaying rank $m$ symmetric tensor fields.

The family of oriented straight lines in ${\mathbb R}^n$ is
parameterized by points of the manifold
    \[
    T{\mathbb S}^{n-1}=\{(x,\xi)\in{\mathbb R}^n\times{\mathbb R}^n\mid |\xi|=1,\langle x,\xi\rangle=0\}\subset{\mathbb R}^n\times{\mathbb R}^n,
    \]
that is, the tangent bundle of the unit sphere ${\mathbb S}^{n-1}$.
A point $(x,\xi)\in T{\mathbb S}^{n-1}$ determines the line
$\{x+t\xi\mid t\in{\mathbb R}\}$.  The Schwartz space ${\mathcal
S}(T{\mathbb S}^{n-1})$ is defined as follows. Given a function
$\varphi\in C^\infty(T{\mathbb S}^{n-1})$, we extend it to some
neighborhood of $T{\mathbb S}^{n-1}$ in ${\mathbb R}^n\times{\mathbb
R}^n$ so that (the extension is again denoted by $\varphi$)
$$
\varphi(x,r\xi)=\varphi(x,\xi)\ (r>0),\quad
\varphi(x+r\xi,\xi)=\varphi(x,\xi)\ (r\in{\mathbb R}).
$$
We say that a function $\varphi\in C^\infty(T{\mathbb S}^{n-1})$
belongs to ${\mathcal S}(T{\mathbb S}^{n-1})$ if the seminorm
$$
\|\varphi\|_{k,\alpha,\beta}=\sup\limits_{(x,\xi)\in T{{\mathbb
S}}^{n-1}}
\left|(1+|x|)^k\partial^\alpha_x\partial^\beta_\xi\varphi(x,\xi)\right|
$$
is finite for every $k\in{\mathbb N}$ and for all multi-indices
$\alpha$ and $\beta$. The family of these seminorms defines the
topology on ${\mathcal S}(T{\mathbb S}^{n-1})$.

The {\it ray transform} $I_m$ is defined for $f=(f_{i_1\dots
i_m})\in{\mathcal S}({\mathbb R}^{n}; S^{m}{\mathbb R}^{n})$ by
\begin{equation}
I_mf (x,\xi)=\int\limits_{-\infty}^\infty f_{i_1\dots i_m}(x+t\xi)\,\xi^{i_1}\dots\xi^{i_m}\,\D  t=\int\limits_{-\infty}^\infty \L f(x+t\xi),\xi^m\R\,\D  t\quad\big((x,\xi)\in T{{\mathbb S}}^{n-1}\big).
                                         \label{1.2}
\end{equation}
Here and henceforth, we use the Einstein summation rule: the summation from 1 to $n$ is assumed over every index repeated in lower and upper positions in a monomial. We use either lower or upper indices for denoting coordinates of vectors and tensors.  Since we work in Cartesian coordinates only, there is no difference between covariant and contravariant tensors.

In the case of an even $m$, the ray transform is the linear continuous operator
$$
I_m:{\mathcal S}({{\mathbb R}}^n;S^m{{\mathbb R}}^n)\rightarrow {\mathcal S}_e(T{\mathbb S}^{n-1}),
$$
and in the case of an odd $m$,
$$
I_m:{\mathcal S}({{\mathbb R}}^n;S^m{{\mathbb R}}^n)\rightarrow {\mathcal S}_o(T{\mathbb S}^{n-1}),
$$
where ${\mathcal S}_e(TS^{n-1})\ ({\mathcal S}_o(TS^{n-1}))$ is the subspace of ${\mathcal S}(TS^{n-1})$
consisting of functions satisfying $\varphi(x,-\xi)=\varphi(x,\xi)$ (satisfying $\varphi(x,-\xi)=-\varphi(x,\xi)$).
To unify these formulas, let us introduce the {\it parity} of $m$
$$
\pi(m)=\left\{\begin{array}{l}e\ \mbox{if}\ m\ \mbox{is even},\\ o\ \mbox{if}\ m\ \mbox{is odd}.\end{array}\right.
$$
Then the ray transform can be initially considered as a linear continuous operator
\begin{equation}
I_m:{\mathcal S}({{\mathbb R}}^n;S^m{{\mathbb R}}^n)\rightarrow {\mathcal S}_{\pi(m)}(T{\mathbb S}^{n-1}).
                          \label{2.2}
\end{equation}

The following theorem is due to Pantjukhina \cite{P}:
\begin{theorem} \label{Th1.2}
Let $n\ge2$ and $m\ge0$. If a function $\varphi\in{\mathcal
S}_{\pi(m)}(T{\mathbb S}^{n-1})$ belongs to the range of the
operator \eqref{2.2}, then for every integer $r\ge0$, there exist
homogeneous polynomials $P^r_{i_1\dots i_m}(x)$ of degree $r$ on
${{\mathbb R}}^n$ such that
\begin{equation}
\int\limits_{\xi^{\perp}}\varphi(x',\xi)\L
x,x'\R^r\,\D x'=P^r_{i_1\dots i_m}(x)\xi^{i_1}\dots\xi^{i_m}\quad
\big((x,\xi)\in T{\mathbb S}^{n-1}\big),
                          \label{2.2a}
\end{equation}
where $dx'$ is the $(n-1)$-dimensional Lebesgue measure on the
hyperplane $\xi^\perp=\{x'\in{{\mathbb R}}^n\mid\L\xi,x'\R=0\}$.

In the case of $n=2$, the converse statement is true: If a function
$\varphi\in{\mathcal S}_{\pi(m)}(T{\mathbb S}^1)$ satisfies
\eqref{2.2a} with some homogeneous polynomials $P^r_{i_1\dots
i_m}(x)$ of degree $r$, then there exists a tensor field
$f\in{\mathcal S}({{\mathbb R}}^2;S^m{{\mathbb R}}^2)$ such that
$\varphi=I_mf$.
\end{theorem}

Our goal, as already mentioned, is to generalize Theorem \ref{Th1.2}
to weighted Sobolev spaces. We need a few notations to state the main
results.

We state the Fourier slice theorem.
    The Fourier transform of symmetric tensor fields
    $$
    F:{\mathcal S}({\mathbb R}^n; S^m{\mathbb R}^n)\to{\mathcal S}({\mathbb R}^n; S^m{\mathbb R}^n),\quad f\mapsto \wh f
    $$
    is defined component wise (hereafter $\I$ is the imaginary unit):
    \[
    \wh{f}_{i_1 \cdots i_{m}}(y)=\frac{1}{(2\pi)^{n/2}}\int \limits_{{\mathbb R}^n} e^{-\I \langle y,x\rangle} f_{i_1 \cdots i_m}(x) \, \D  x.
    \]
    The Fourier transform $F:{\mathcal S}(T{\mathbb S}^{n-1})\to{\mathcal S}(T{\mathbb S}^{n-1}),\ \vp\mapsto\wh\vp$ is defined as the $(n-1)$-dimensional Fourier transform over the subspace $\xi^{\perp}$:
    \[
    \wh{\vp}(y,\xi)=\frac{1}{(2\pi)^{(n-1)/2}}\int\limits_{\xi^{\perp}} e^{-\I \L y,x\rangle} \vp(x,\xi) \, \D  x\quad\big((y,\xi)\in T{{\mathbb S}}^{n-1}\big).
    \]
    The Fourier slice theorem \cite[formula (2.1.5)]{Sharafutdinov:Book} states:
    \Beq
    \wh{If}(y,\xi)=\sqrt{2\pi}\L \wh{f}(y), \xi^{m}\rangle \mbox{ for } (y,\xi)\in T{\mathbb S}^{n-1}.
                                  \label{2.3}
    \Eeq

We recall that ${\mathcal S}_{\mathrm{sol}}({{\mathbb
R}}^n;S^m{{\mathbb R}}^n)\ (m\ge1)$ is the subspace of ${\mathcal
S}({{\mathbb R}}^n;S^m{{\mathbb R}}^n)$ consisting of {\it
solenoidal} tensor fields satisfying
    \begin{equation}
        \sum\limits_{p=1}^n\frac{\partial f_{pi_2\dots i_m}}{\partial x^p}=0.
                                        \label{2.10}
    \end{equation}
For $m=0$ we set ${\mathcal S}_{\mathrm{sol}}({\mathbb R}^n)
={\mathcal S}({\mathbb R}^n)$.

For an integer $r\ge0$, real $s$ and $t>-(n-1)/2$, the Hilbert space
$H^{(r,s)}_t(T{{\mathbb S}}^{n-1})$ was introduced in
\cite[Definition 3.4]{HORF_Work}. Roughly speaking, the space
consists of functions $\varphi(x,\xi)$ on $T{{\mathbb S}}^{n-1}$
with quadratically integrable derivatives of order $\le r$ with
respect to $\xi$ and with quadratically integrable derivatives of
order $\le s$ with respect to $x$. For an integer $r\ge0$, real $s$
and $t>-n)/2$, the Hilbert space
$H^{(r,s)}_{t,\mathrm{sol}}({{\mathbb R}}^n;S^m{{\mathbb R}}^n)$ was
introduced in \cite[Definition 5.2]{HORF_Work}. The interpretation
of solenoidal tensor fields belonging to
$H^{(r,s)}_{t,\mathrm{sol}}({{\mathbb R}}^n;S^m{{\mathbb R}}^n)$ is
not so easy; nevertheless, these spaces inherit basic properties of
standard Sobolev spaces. In the next section, definitions of spaces
$H^{(r,s)}_t(T{{\mathbb S}}^{n-1})$ and
$H^{(r,s)}_{t,\mathrm{sol}}({{\mathbb R}}^n;S^m{{\mathbb R}}^n)$
will be reproduced with some simplifications in the 2D case that is
of our main interest in the current work. Moreover, these spaces
will be defined for any real $r$ in the 2D case.

    By \cite[Theorem 1.1]{HORF_Work}, for all $n\ge2$ and $m\ge0$, the ray transform
    $$
    I_m:{\mathcal S}_{\mathrm{sol}}({{\mathbb R}}^n;S^m{{\mathbb R}}^n)\rightarrow {\mathcal S}_{\pi(m)}(T{{\mathbb S}}^{n-1})
    $$
    extends to the isometric embedding of Hilbert spaces
    \begin{equation}
        I_m:H^{(r,s)}_{t,\mathrm{sol}}({{\mathbb R}}^n;S^m{{\mathbb R}}^n)\rightarrow H^{(r,s+1/2)}_{t+1/2,\pi(m)}(T{{\mathbb S}}^{n-1})
                                               \label{2.11}
    \end{equation}
    for every integer $r\ge0$, every real $s$ and every $t>-n/2$. See \eqref{2.2} for the additional index $\pi(m)$ on the right-hand side of
    \eqref{2.11}.
    One of the goals of this paper is to generalize this isometry result for any real $r$ in the 2D case.
Next, as a consequence of this isometry result, we prove a range
characterization theorem.

We are now ready to state main results of the paper. The spaces
appearing in the following statements are defined in the next
section.

\begin{theorem}[Reshetnyak formula] \label{Th2.1}
 For real $r,s,t > -1$ and for any $f \in
\Sc_{\mathrm{sol}}(\Rb^2; S^m\Rb^2)$, the following $r^{th}$-order
Reshetnyak formula holds:
\begin{equation}
\|f\|_{H^{r,s}_{t,\mathrm{sol}}(\Rb^2; S^m \Rb^2)} = \|I_m
f\|_{H^{r,s+1/2}_{t+1/2,\pi(m)}(T\Sb^1)}.
                   \label{2.12}
\end{equation}
\end{theorem}

\begin{theorem}[Range characterization] \label{Th2.2}
 For any integer $m\ge0$, any real $r,s$
and any $t>-1$, the operator $I_m: H^{r,s}_{t,\mathrm{sol}}(\Rb^2;
S^m \Rb^2) \to H^{r,s+1/2}_{t+1/2,\pi(m)}(T\Sb^1)$ is a bijective
isometry of Hilbert spaces.
\end{theorem}

\section{The Reshetnyak formula}

In the 2D case, it is convenient to represent a tensor field $f \in
\Sc(\Rb^2; S^m\Rb^2)$ as $f=(f_0,\dots,f_m)$, where $f_j\in
\Sc(\Rb^2)$ are defined by
\begin{equation}
f_j=f_{\underbrace{\scriptstyle 1\dots
1}_{m-j}\underbrace{\scriptstyle 2\dots 2}_j}.
                 \label{Eq1.1A}
\end{equation}
It is also convenient to assume that $f_j=0$ for $j>m$.

The manifold $T\Sb^1$ is parameterized by $(p,\theta)\in\Rb
\times[0,2\pi)$, i.e., a point $(x,\xi)\in T\Sb^1$ is defined by $x=
p(-\sin\theta,\cos\theta), \xi=(\cos\theta,\sin\theta)$. The ray
transform is a bounded linear operator $I_{m}: \Sc(\Rb^2;
S^{m}\Rb^2) \to \Sc(T\Sb^1)$  defined by
\[
I_m f (p,\theta) = \int\limits_{\Rb} \sum\limits_{j=0}^m
\binom{m}{j} f_j (-p\sin\theta+t\cos\theta, p\cos\theta+
t\sin\theta) \cos^{m-j}\theta \sin^j\theta\, \D t.
\]
The definition implies that $I_m f(-p, \theta+\pi) = (-1)^m I_m f
(p,\theta)$.

Let $\Sc_{\mathrm{sol}}(\Rb^2; S^m \Rb^2)$ denote the space of
solenoidal tensor fields whose components belong to the Schwartz
space. A tensor field $f=(f_0,\dots,f_m)$ is solenoidal iff
\Beq
\frac{\PD f_j}{\PD x}(x,y) + \frac{\PD f_{j+1}}{\PD y} (x,y) = 0
\quad (0 \leq j \leq m).
                       \label{Eq1.1}
\Eeq

In terms of the Fourier transform $\widehat f=(\widehat
f_0,\dots,\widehat f_m)$, \eqref{Eq1.1} is written as
\[
\cos \theta\,\wh{f}_j(q,\theta)+\sin \theta\,\wh{f}_{j+1}(q,\theta)
= 0 \quad (0\leq j\leq m, q>0),
\]
where $(q,\theta)$ are polar coordinates in the Fourier space. From
this we obtain by induction in $m-j$
\Beq
\cos^{m-j}\theta\,\wh{f}_j(q,\theta)=(-1)^{m-j}\sin^{m-j}\theta\,\wh{f}_m(q,\theta)
\quad (0\leq j\leq m).
               \label{Eq1.2AA}
\Eeq Hence
\Beq
\sin^{m-j}\theta\,\wh{f_j}(q,\theta+\pi/2)
=\cos^{m-j}\theta\,\wh{f}_m(q,\theta+\pi/2) \quad (0\leq j\leq m).
                      \label{Eq1.2A}
\Eeq

The Fourier transform of $\psi \in \Sc(T\Sb^1)$ is defined by
\Beq
\wh{\psi} (q,\theta) = \frac{1}{\sqrt{2\pi}} \int\limits_{\Rb}
e^{-\I q p} \psi(p,\theta)\, \D p.
                   \label{Eq1.2}
\Eeq
The  Fourier slice theorem for the ray transform of scalar
functions is written in polar coordinates as follows: \Beq
\wh{I_0f}(q,\theta) = \sqrt{2\pi}\,\wh f(q,\theta+\pi/2) \quad
\mbox{for} \quad f\in{\mathcal S}({\mathbb R}^n).
                   \label{Eq1.2a}
\Eeq

The following version of the Fourier slice theorem is valid for the
ray transform of solenoidal tensor fields.

\begin{lemma}
For $f\in\Sc_{\mathrm{sol}}(\Rb^2; S^m \Rb^2)$,
\begin{align*}
\sin^m \theta\, \wh{I_m f}(q,\theta) = \wh{f}_m(q,\theta+\pi/2)
\quad \text{for } q >0.
\end{align*}
\end{lemma}

\bpr Applying \eqref{Eq1.2} to $I_m f$ and using \eqref{Eq1.2a}, we
get
$$
\begin{aligned}
\wh{I_m f} (q,\theta)&= \frac{1}{\sqrt{2\pi}} \sum\limits_{j=0}^m {m\choose j} \cos^{m-j}\theta\sin^j\theta\, \wh{I_0 f_j}(q,\theta) \\
&= \sum\limits_{j=0}^m {m\choose j}\cos^{m-j}\theta\sin^j\theta
\,\wh{f_j} (q,\theta+\pi/2).
\end{aligned}
$$
From this we derive with the help of \eqref{Eq1.2A} $$
\begin{aligned}
\sin^m\theta\, \wh{I_m f} (q,\theta)&= \sum\limits_{j=0}^m {m\choose j}
\cos^{m-j}\theta \sin^{2j}\theta \Big(\sin^{m-j}\theta\, \wh{f}_j(q,\theta+\pi/2)\Big)\\
&= \sum\limits_{j=0}^m {m\choose j}
\cos^{m-j}\theta \sin^{2j}\theta \cos^{m-j}\theta\, \wh{f}_{m}(q,\theta+\pi/2)\Big)\\
 &=\wh{f}_{m}(q,\theta+\pi/2)\sum\limits_{j=0}^m {m\choose j} \cos^{2(m-j)} \theta \sin^{2j} \theta \\
 &=\wh{f}_{m}(q,\theta+\pi/2).
\end{aligned}
$$
\epr

For a function $\psi \in \Sc(T\Sb^1)$, we consider the Fourier
series expansion
\[
\psi(p,\theta) = \sum\limits_{l=-\infty}^\infty \psi_l(p) e^{\I l \theta},
\]
where
\[
\psi_l (p) =\frac{1}{2\pi} \int\limits_0^{2\pi} \psi(p,\theta) e^{-\I l \theta} \D \theta.
\]
If $\psi(p, \theta) = I_m f(p,\theta)$ for some $f \in \Sc(\Rb^2; S^m \Rb^2)$, then
\Beq\label{Eq1.8}
    \psi_l(-p) = (-1)^m e^{\I l \pi} \psi_l (p) = (-1)^{m+l} \psi_l(p).
\Eeq
The function $\wh{\psi}(q,\theta)$ has the Fourier series expansion
\[
\wh{\psi}(q,\theta) = \sum\limits_{l=-\infty}^\infty \wh{\psi_l} (q) e^{\I l \theta},
\]
where $\wh{\psi_l}$ are the usual one-dimensional Fourier transforms of $\psi_l$ and if $\psi=I_mf$, then the Fourier coefficients of the Fourier transform also satisfy
\[
\wh{\psi_l}(-q) = (-1)^{m+l} \wh{\psi_l}(q).
\]
With this in mind, let us define the space
\[
\Sc_{\pi(m)} (T\Sb^1) \coloneqq \Big{\{} \phi \in \Sc(T\Sb^1): \phi(-p,\theta+\pi) = (-1)^m \phi(p,\theta) \Big{\}}.
\]

We are now ready to define the Sobolev spaces.
\begin{definition}
    For real $r,s$ and $t > -1/2$, the space $H^{r,s}_{t,\pi(m)}(T\Sb^1)$ is the completion of $\Sc_{\pi(m)}(T\Sb^1)$ with respect to the norm
    \[
    \|\psi\|^2_{H^{r,s}_{t,\pi(m)}(T\Sb^1)} = \frac{1}{4\pi} \sum\limits_{l=-\infty}^\infty (1+l^2)^r \int\limits_{\Rb} |q|^{2t} (1+q^2)^{s-t} |\wh{\wt{\psi}_l}(q) |^2 \D q,
    \]
    where $\wt{\psi}(p,\theta) = \sin^m \theta\, \psi(p,\theta)$.
\end{definition}
Henceforth we use the notation
$\wt{\psi}(p,\theta)=\sin^{m}\theta\,\psi(p,\theta)$ for a function
$\psi\in{\mathcal S}(T{\mathbb S}^1)$. We note that $\wt{\cdot}$ and
$\wh{\cdot}$ commute.

\begin{definition}
    For real $r,s$ and $t > -1$, the space $H^{r,s}_{t,\mathrm{sol}}(\Rb^2;S^m\Rb^2)$ is the completion of $\Sc_{\mathrm{sol}}(\Rb^2;S^m\Rb^2)$ with respect to the norm
    \[
    \|f\|^2_{H^{r,s}_{t,\mathrm{sol}}(\Rb^2;S^m\Rb^2)} = \frac{1}{2\pi} \sum\limits_{l=-\infty}^\infty (1+l^2)^r \int\limits_0^\infty p^{2t+1} (1+p^2)^{s-t} |(\wh{f}_m)_l(p)|^2 \D p.
    \]
\end{definition}
In the definition, $f_{m}$ denotes the last component of $f$ as in
\eqref{Eq1.1A}. Due to \eqref{Eq1.2AA}, this is indeed a norm.

We next rewrite the Fourier slice theorem for solenoidal fields in terms of Fourier coefficients. We begin by writing $\wh{f}_m$ in terms of its Fourier coefficients:
\[
\wh{f}_m (q,\theta) = \sum\limits_{l=-\infty}^\infty (\wh{f}_m)_l (q) e^{\I l \theta}.
\]
This implies
\[
\wh{f}_m(q,\theta + \pi/2) = \sum\limits_{l=-\infty}^\infty \I^l (\wh{f}_m)_l (q) e^{\I l \theta}.
\]
Next, for $\psi \in \Sc_{\pi(m)}(T\Sb^1)$,
\begin{align*}
    \sin^m \theta\, \psi(p,\theta) &= \sum\limits_{l=-\infty}^\infty \psi_l(p) e^{\I l \theta} \frac{(e^{2\I \theta} - 1)^m}{(2\I)^m e^{\I m \theta}} \\
    &= \frac{1}{(2\I)^m} \sum\limits_{l=-\infty}^\infty \psi_l(p) e^{\I (l-m) \theta} \sum\limits_{k=0}^m \binom{m}{k} (-1)^k e^{2\I (m-k) \theta} \\
    &= \frac{1}{(2\I)^m} \sum\limits_{l=-\infty}^\infty \sum\limits_{k=0}^m (-1)^k \binom{m}{k} \psi_l(p) e^{\I (l+m-2k)\theta} \\
    &= \sum\limits_{l=-\infty}^\infty \lb \frac{1}{(2\I)^m} \sum\limits_{k=0}^{m}(-1)^k \binom{m}{k} \psi_{l-m+2k}(p) \rb e^{\I l \theta}.
\end{align*}
This gives an expression for the Fourier coefficients of
$\wt{\psi}$. Since $\psi\in \Sc_{\pi(m)}(T\Sb^1)$, its Fourier
coefficients satisfy $\wt{\psi}_l (-p) = (-1)^l \wt{\psi}_l (p)$.
Substituting $\psi=I_{m}f$, we write the Fourier slice theorem in
the form
\begin{align*}
    \sum\limits_{l=-\infty}^\infty \sum\limits_{k=0}^{m}\lb \frac{1}{(2\I)^m} (-1)^k {m\choose k} \lb \wh{I_m f}\rb_{l-m+2k}(q) \rb e^{\I l \theta} = \sum\limits_{l=-\infty}^\infty \lb \I^l (\wh{f_m})_l (q) \rb e^{\I l \theta}.
\end{align*}
From this we have the following: For $q>0$,

\[
\sum\limits_{k=0}^{m}\frac{1}{(2\I)^m} (-1)^k {m\choose k} \wh{I_m f}_{l-m+2k}(q) = \I^{l} \lb\wh{f_{m}}\rb_{l}(q).
\]

\begin{proof}[Proof of Theorem \ref{Th2.1}]
For  $f \in \Sc_{\mathrm{sol}}(\Rb^2; S^m\Rb^2)$,
    \begin{align*}
        \|I_m f\|_{H^{r,s+1/2}_{t+1/2,\pi(m)}(T\Sb^1)}^2& = \frac{1}{4\pi} \sum\limits_{l=-\infty}^\infty (1+l^2)^r \int\limits_{\Rb} |q|^{2t+1} (1+q^2)^{s-t} |\wh{(\widetilde{I_m f})_l}(q)|^2 \D q \\
        &= \frac{1}{4\pi} \sum\limits_{l=-\infty}^\infty (1+l^2)^r \int\limits_{\Rb} |q|^{2t+1} (1+q^2)^{s-t} \left | \frac{1}{2^m} \sum\limits_{k=0}^m (-1)^k \binom{m}{k} (\wh{I_m f})_{l-m+2k} (q)  \right |^2 \D q \\
        &= \frac{2}{4\pi} \sum\limits_{l=-\infty}^\infty (1+l^2)^r \int\limits_{0}^\infty |q|^{2t+1} (1+q^2)^{s-t} \left | \frac{1}{2^m} \sum\limits_{k=0}^m (-1)^k \binom{m}{k} (\wh{I_m f})_{l-m+2k} (q)  \right |^2 \D q\\
        &=\frac{1}{2\pi} \sum\limits_{l=-\infty}^\infty (1+l^2)^r \int\limits_{0}^\infty |q|^{2t+1} (1+q^2)^{s-t}\lvert (\wh{f_{m}})_{l}(q)\rvert^2 \, \D q\\
        &=\|f\|^2_{H^{r,s}_{t,\mathrm{sol}}(\Rb^2;S^m\Rb^2)}.
    \end{align*}
\end{proof}

\section{Range characterization}

Before proving Theorem \ref{Th2.2}, we present two auxiliary
statements.

Let $\Sc_{\pi(m),0}(T\Sb^1)$ be the subspace of
$\Sc_{\pi(m)}(T\Sb^1)$ consisting of functions $\psi$ satisfying
\[
\wh{\psi}(q,\theta) = 0 \quad \text{~for $|q| \leq \epsilon$ with
some $\epsilon = \epsilon(\psi) > 0$.}
\]

\begin{lemma}\label{density} The space
    $\Sc_{\pi(m),0}(T\Sb^1)$ is dense in $H^{r,s}_{t,\pi(m)}(T\Sb^1)$ for all $r,s$ and $t > -1/2$.
\end{lemma}
\begin{proof}
    The proof follows along the same lines as \cite[Lemma~4.2]{Sh5}; we give the proof here for the sake of completeness.

    By definition, $\Sc_{\pi(m)}(T\Sb^1)$ is dense in $H^{r,s}_{t,\pi(m)}(T\Sb^1)$. We show that each $\phi \in \Sc_{\pi(m)}(T\Sb^1)$ can be approximated by functions from $\Sc_{\pi(m),0}(T\Sb^1)$ in the norm of $H^{r,s}_{t,\pi(m)}(T\Sb^1)$.

    Choose a smooth even function $\mu: \Rb \to \Rb$ such that $\mu(q) = 0$ for $|q| \leq 1$, $\mu(q) = 1$ for $|q| \geq 2$ and $0 \leq \mu(q) \leq 1$ for all $q$. Given $\phi \in \Sc_{\pi(m)}(T\Sb^1)$, define $\psi^k$, for $k=1,2,\dots$ by
    \[
    \wh{\psi^k}(q,\theta) = \mu(kq) \wh{\phi}(q,\theta).
    \]
    Clearly $\psi^k \in \Sc_{\pi(m),0}(T\Sb^1)$. We now show that
    \[
    \|\psi^k - \phi\|_{H^{r,s}_{t,\pi(m)}(T\Sb^1)} \to 0 \quad \text{as~} k \to \infty.
    \]
    Let $\wh{\phi}$ have the expansion
    \[
    \wh{\phi}(q,\theta) = \sum\limits_{l=-\infty}^{\infty} \wh{\phi}_l(q) e^{\I l \theta}.
    \]
    Then
    \begin{align*}
        \wh{\psi}^k (q,\theta) - \wh{\phi}(q,\theta) = \sum\limits_{l=-\infty}^\infty \lb \mu(k q) - 1 \rb \wh{\phi_l}(q) e^{\I l \theta},
    \end{align*}
    and
    \begin{align*}
        \wt{\wh{\psi^k} (q,\theta) - \wh{\phi}(q,\theta)} &= \sum\limits_{l=-\infty}^\infty \lb \mu(k q) - 1 \rb \wh{\wt{\phi}}_l(q) e^{\I l \theta}.
    \end{align*}
    Here we note that $\wt{\cdot}$ denotes multiplication by $\sin^{m}\theta$ and since the Fourier transform is only applies to the $p$ variable in $T\Sb^{1}$, the Fourier transform $\wh{\cdot}$ and $\wt{\cdot}$ commute.
    The right hand side vanishes for $|q| \geq 2/k$. By definition of the norm,
    \begin{align*}
        \|\psi^k - \phi\|^2_{H^{r,s}_{t,\pi(m)}(T\Sb^1)} &= \frac{1}{4\pi} \sum\limits_{l=-\infty}^\infty (1+l^2)^r \int\limits_{-2/k}^{2/k} |q|^{2t} (1+q^2)^{s-t} \lb \mu(k q) - 1 \rb^2 |\wh{\tilde{\phi}}_l(q)|^2 \D q \\
        &\leq \frac{1}{4\pi} \sum\limits_{l=-\infty}^\infty (1+l^2)^r \int\limits_{-2/k}^{2/k} |q|^{2t} (1+q^2)^{s-t} |\wh{\tilde{\phi}}_l(q)|^2 \D q  \to 0 \mbox{ as } k \to \infty.
    \end{align*}
    This completes the proof.
\end{proof}
The next lemma is similar to a claim made as part of
\cite[Lemma~4.2]{Sh5} as well. We again give the proof for the sake
of completeness.
\begin{lemma}\label{existence}
    If $\psi \in \Sc_{\pi(m),0}(T\Sb^1)$, then there exists $f \in \Sc(\Rb^2)$ such that
    \[
    \I^l \wh{f}_l (q) = \wh{\wt{\psi_l}}(q) \quad \text{for } q > 0.
    \]
\end{lemma}
\begin{proof}
We define the following subspace:
\[
    \Sc_{e,0}(T\Sb^1) \coloneqq \{ \phi \in \Sc_{\pi(m),0}(T\Sb^1): \phi(-p,\theta+\pi) = \phi(p,\theta)\}.
    \]
    Since $\wt{\cdot}$ involves multiplication by $\sin^{m}\theta$, we see that if  $\psi \in \Sc_{\pi(m),0}(T\Sb^1)$, then $\wt{\psi} \in \Sc_{e,0}(T\Sb^1)$.

   Next, let the Fourier series expansion of $\wt{\psi}$ be
   \[
   \wt{\psi}(p,\theta)=\sum\limits_{l=-\infty}^{\infty} \wt{\psi}_{l}(p) e^{\I l \theta}.
   \]
   Now define a function $f$ such that its Fourier transform $\wh{f}$ has for its Fourier coefficients:
   \[
   \wh{f}_{l}(q)=(-\I)^{l}\wh{\wt{\psi}}_{l}(q) \mbox{ for } q>0.
   \]
    In other words, define the function $\wh{f}$ by the series
    \[
    \wh{f}(z) = \sum\limits_{l=-\infty}^\infty (-\I)^{l}\wh{\wt{\psi}}_{l}(|z|) e^{\I l z/|z|}.
    \]
    Since each $\wh{\wt{\psi}}_l$ vanishes near $0$ and decreases rapidly at $\infty$, $\wh{f}$ and hence $f$ belongs to $\Sc(\Rb^2)$.
\end{proof}

\begin{proof}[Proof of Theorem \ref{Th2.2}]
Due to the Reshetnyak formula \eqref{2.12}, the range of the
operator
$$
I_m: H^{r,s}_{t,\mathrm{sol}}(\Rn) \to
H^{r,s+1/2}_{t+1/2,\pi(m)}(T\Sb^1)
$$ is a closed subspace. It remains to prove that $I_m$ is a
surjective operator.

Let $\phi \in H^{r,s+1/2}_{t+1/2,\pi(m)}(T\Sb^1)$ be orthogonal to
the range of $I_m$. In particular,
\[
\langle I_m f, \phi \rangle_{H^{r,s+1/2}_{t+1/2, \pi(m)}(T\Sb^1)} = 0 \quad \mbox{for all } f \in \Sc_{\mathrm{sol}}(\Rb^2; S^m\Rb^2).
\]
Let us choose a sequence $\{\phi_p\} \in \Sc_{\pi(m)}(T\Sb^1)$ converging to $\phi$ in $H^{r,s+1/2}_{t+1/2,\pi(m)}(T\Sb^1)$.
Such a sequence exists by the definition of the space $H^{r,s+1/2}_{t+1/2,\pi(m)}(T\Sb^1)$.
Then the sequence of norms $\|\phi_p\|_{H^{r,s+1/2}_{t+1/2,\pi(m)}(T\Sb^1)}$ is bounded and
\[
A_p \coloneqq \langle I_m f, \phi_p \rangle_{H^{r,s+1/2}_{t+1/2,\pi(m)}(T\Sb^1)} \to 0 \text{~as~} p \to \infty,
\text{~for any~} f \in \Sc_{\mathrm{sol}}(\Rb^2;S^m\Rb^2).
\]
Using the definition of the inner product, we get
\begin{align*}
    A_p = \frac{1}{4\pi} \sum\limits_{l=-\infty}^\infty (1+l^2)^r \int\limits_{\Rb} |q|^{2t+1} (1+q^2)^{s-t} \wh{(\widetilde{I_m f})_l}
    \overline{\wh{(\wt{\phi}_p)_l}}\, \D q,
\end{align*}
where, as before, tildes denote multiplication by $\sin^m \theta$. Note that the integral becomes twice of that over the positive reals. Using the Fourier slice theorem in terms of Fourier coefficients,
\begin{align*}
    A_p &= \frac{1}{2\pi} \sum\limits_{l=-\infty}^\infty (1+l^2)^r \int\limits_0^\infty |q|^{2t+1} (1+q^2)^{s-t} \I^l
    (\wh{f_m})_l (q)  \overline{\wh{(\wt{\phi}_p)_l}} \, \D q.
\end{align*}
Using Lemma~\ref{existence},
\[
\frac{1}{2\pi} \sum\limits_{l=-\infty}^\infty (1+l^2)^r
\int\limits_0^\infty |q|^{2t+1} (1+q^2)^{s-t} \wh{\tilde{\psi}}_l
(q) \overline{\wh{(\wt{\phi}_p})_l}\, \D q \to 0 \text{~as~} p \to
\infty,
\]
for any $\psi \in \Sc_{\pi(m),0}(T\Sb^1)$. Since this space is dense in $H^{r,s+1/2}_{t+1/2,\pi(m)}(T\Sb^1)$, we conclude
\[
\phi_p \rightharpoonup 0 \text{~in~}
H^{r,s+1/2}_{t+1/2,\pi(m)}(T\Sb^1)\text{~as~} p \to \infty.
\]
But $\phi_p$ was chosen such that $\phi_p \to \phi$ in
$H^{r,s+1/2}_{t+1/2,\pi(m)}(T\Sb^1)$. This yields that $\phi \equiv
0$. Hence, the orthogonal complement of the range of $I_m$ is equal
to zero and thus $I_m$ is surjective.
\end{proof}

\end{document}